\documentclass[11pt]{article}
\usepackage{amsmath}
\usepackage{amsfonts}
\usepackage{amssymb}

\usepackage{graphicx}
\usepackage{graphics}

\newtheorem{theorem}{Theorem}[section]
\newtheorem{corollary}{Corollary}[section]
\newtheorem{lemma}{Lemma}[section]
\newtheorem{proposition}{Proposition}[section]
\numberwithin{equation}{section}
\newtheorem{definition}{Definition}[section]

\newtheorem{remark}{Remark}[section]

\usepackage[top=1.5cm, bottom=1.5cm, left=1cm, right=1cm]{geometry}

\font\QEDlogofont=msam10 at 10pt

\def\QEDblogo{\hbox{\QEDlogofont\char'004}}
\newif\ifnologo\nologofalse
\newif\iflogo
\newif\ifblogo\blogofalse
\newif\iftopprhead\topprheadfalse

\def\prooffont{\normalsize}
\newenvironment{proof}{\par\addvspace{6pt plus2pt}%\topprheadtrue\par\iftopprhead\vspace*{-13pt}\else%
\par%\addvspace{6.5pt plus2pt}%\fi%
\noindent\prooffont{\bf\em Proof:}\hskip6pt\ignorespaces}{%
   \ifblogo\hskip1.2pt
            \blacksquare
   \else
   \ifnologo
   \else
   \hfill
            \QEDblogo
   \fi\fi
\par\addvspace{6pt plus2pt}\global\topprheadfalse}%

\usepackage{lipsum}

\newcommand\blfootnote[1]{%
  \begingroup
  \renewcommand\thefootnote{}\footnote{#1}%
  \addtocounter{footnote}{-1}%
  \endgroup
}

\begin{document}
%%%

\begin{center}

\begin{large}
{\bf Generalized Extended Riemann-Liouville type fractional derivative operator}
\end{large}
\vspace{10pt}

{\bf H. Abbas $^{\rm a}$,  A. Azzouz $^{\rm b,}$\footnote{ Corresponding Author},   M. B. Zahaf $^{\rm
c}$ and M. Belmekki $^{\rm
d}$}\vspace{6pt}
\\
\vspace{6pt} $^{\rm a}${\em Department of Mathematic, Faculty of Sciences. \break University Dr Tahar Moulay. P. Box 138 Nasr. Saida. Algeria 20000.
}\\
\vspace{6pt} $^{\rm b}${\em Department of Common studies, Faculty of Technology. \break University Dr Tahar Moulay. P. Box 138 Nasr. Saida. Algeria 20000.
}\\
 \vspace{6pt} $^{\rm c}${\em Laboratoire d'Analyse Non Lin\'eaire et Math\'ematiques Appliqu\'ees,\break
Universit\'e de Tlemcen, BP 119,  13000-Tlemcen, Algeria .
}\\
\vspace{6pt} $^{\rm d}${\em High School of Applied Sciences, \break P. Box 165 RP. Bel Horizon, 13000-Tlemcen, Algeria.}\\

\blfootnote{E-mail adresses: a.hafida@yahoo.fr (H.Abbas),  abdelhalim.azzouz.cus@gmail.com (A.Azzouz), m\_b\_zahaf@yahoo.fr (M.B.Zahaf), m.belmekki@yahoo.fr (M. Belmekki)}
\end{center}
\vspace{1cm}

\begin{abstract}
In this paper, we aim to present  new extensions of incomplete gamma, beta, Gauss hypergeometric, confluent
hypergeometric function and Appell-Lauricella hypergeometric functions, by using the extended Bessel function due to Boudjelkha \cite{Boudjelkha}. Some recurrence relations, transformation formulas, Mellin transform and integral representations are obtained for these generalizations. Further, an extension of the Riemann-Liouville fractional derivative operator is established.
\end{abstract}


\textbf{MSC subject}:  26A33, 33B15, 33B20, 33C20, 33C65.

\section{Introduction}
In recent years, incomplete gamma functions have been used in many  problems in applied mathematics, statistics, engineering and many other fields including physics and biology. Most generally, special functions became powerful tools to treat all these areas. Classical gamma and Euler's beta functions are defined by
\begin{eqnarray}
\gamma(\alpha, x)&=&\int_{0}^{x}t^{\alpha-1}e^{-t}dt,\quad( \Re(\alpha)>0),\label{chaud1}\\
\Gamma(\alpha, x)&=&\int_{x}^{\infty}t^{\alpha-1}e^{-t}dt,\label{chaud2}\\
B(x,y)&=&\int_{0}^{1}t^{x-1}(1-t)^{y-1}dt, \quad (\Re(x)>0,\; \Re(y)>0).\label{chaud3}
\end{eqnarray}
 Using an exponential regulazing term, Chaudhry et al.\cite{Chaudhry} extended the incomplete gamma function as follows
\begin{eqnarray}
\gamma(\alpha, x; p)&=&\int_{0}^{x}t^{\alpha-1}e^{-t-\frac{p}{t}}dt,\quad ( \Re(p)>0;\;p=0,\,  \Re(\alpha)>0),\label{chaud4}\\
\Gamma(\alpha, x; p)&=&\int_{x}^{\infty}t^{\alpha-1}e^{-t-\frac{p}{t}}dt.\label{chaud5}
\end{eqnarray}
 They proved the following recurrence formula
\begin{equation*}
    \gamma(\alpha, x; p)+\Gamma(\alpha, x; p)=2p^{\alpha /2}K_{\alpha}(2\sqrt{p}),\quad ( \Re(p)>0),
\end{equation*}
where $K_{\alpha}(z)$ is the Macdonald function, known also as modified bessel function of the third kind, defined for any $Re(z)>0$ by
\begin{equation*}
    K_{\alpha}(z)=\frac{(z/2)^{\alpha}}{2}\int_{0}^{\infty}t^{-\alpha-1}e^{-t-z^2/4t}dt.
\end{equation*}

A first extension of Euler's beta function is given by Chaudhry et al. \cite{Chaudqadir} as follows
\begin{eqnarray}
B(x,y,p)&=&\int_{0}^{1}t^{x-1}(1-t)^{y-1}e^{\frac{-p}{t(1-t)}}dt,\;(\Re(p)>0;\;p=0,\;\Re(x)>0,\; \Re(y)>0). \label{chaudqadir}
\end{eqnarray}

These extensions are useful and provide new connections with error and Whittaker functions. For $p=0$, \eqref{chaud4}, \eqref{chaud5} and \eqref{chaudqadir}  will be reduced to known incomplete gamma and beta functions \eqref{chaud1}, \eqref{chaud2} and \eqref{chaud3}, respectively. Instead of using the exponential function, Chaudhry and Zubair \cite{Chaudhry2} proposed a generalized extension of \eqref{chaud4}, \eqref{chaud5} in the following form
 \begin{eqnarray}
\gamma_\mu(\alpha, x; p)&=&\sqrt{\frac{2p}{\pi}}\int_{0}^{x}t^{\alpha-\frac{3}{2}}e^{-t}K_{\mu+\frac{1}{2}}\left(\frac{p}{t}\right)dt,\label{chaudzubair1}\\
\Gamma_\mu(\alpha, x; p)&=&\sqrt{\frac{2p}{\pi}}\int_{x}^{\infty}t^{\alpha-\frac{3}{2}}e^{-t}K_{\mu+\frac{1}{2}}\left(\frac{p}{t}\right)dt,\;(\Re(x)>0,\;\Re(p)>0,\;-\infty<\alpha<\infty).
\label{chaudzubair2}
\end{eqnarray}

 Nowadays, many authors are developing new extensions of Euler's gamma, beta and hypergeometric functions based on the paper of Chaudhry and Zubair \cite{Chaudhry2} by considering exponential and certain modified special functions (see for more details \cite{Harris,goswami,Luo,ozregin2010, ozregin2011}). Very recently,  Agarwal et al. \cite{agarwal} developed an extension of the Euler's beta function as follows
 \begin{equation}\label{agarwal}
B_\mu(x, y; p;m)=\sqrt{\frac{2p}{\pi}}\int_{0}^{1}t^{x-\frac{3}{2}}(1-t)^{y-\frac{3}{2}}K_{\mu+\frac{1}{2}}\left(\frac{p}{t^m(1-t)^m}\right)dt,
\end{equation}
where $x, y \in\mathbb{C}$, $m > 0$ and $\Re(p) > 0$.

 In the present paper, we introduce  new generalized incomplete gamma and Euler's beta functions by substituting in \eqref{chaudzubair1}, \eqref{chaudzubair2} and \eqref{agarwal} the Macdonald function $K_{\alpha}(z)$ by it's extended one developed by Boudjelkha \cite{Boudjelkha}, namely
\begin{equation}\label{boujlkha}
R_K(z,\alpha,q,\lambda)=\frac{(z/2)^{\alpha}}{2}\int_{0}^{\infty}t^{-\alpha-1}\frac{e^{-qt-z^2/4t}}{1-\lambda e^{-t}}dt,
\end{equation}
where $|\arg z^2|<\pi/2,$ $0<q\leq 1$ and  $-1\leq \lambda\leq 1$.\\
Clearly, when $\lambda = 0$ and $q=1$, $R_K(z,\alpha,q,\lambda)$ is reduced to  $K_{\alpha} (z)$. Moreover, Boudjelkha proved that the $R_K(z,-\alpha,q,\lambda)$ function can be expanded in terms of $K_{\alpha} (z)$ as follows
\begin{equation}
R_K(z,-\alpha,q,\lambda)=\sum_{n=0}^{\infty}\lambda^n\frac{K_\alpha(z\sqrt{q+n})}{(q+n)^{\alpha/2}},\;\Re(z^2)>0,\;0<q\leq1,\;-1\leq \lambda\leq 1,
\end{equation}
and showed that the  behavior of the function $R_K(z,-\alpha,q,\lambda)$ for small values of $z$ is described by the asymptotic formulas:
\begin{eqnarray}
R_K(z,-\alpha,q,\lambda)\sim \left\{
\begin{array}{ll}
\frac{1}{2}\frac{\Gamma(-z)}{(z/2)^{-\alpha}}(1-\lambda)^{-1},\;z\to 0,-1<\lambda<1,&\Re(\alpha)<0,\\
\frac{1}{2}\frac{\Gamma(z)}{(z/2)^{\alpha}}\Phi(\lambda,\alpha,q),\; z\to 0,\; -1\leq\lambda\leq 1,&\Re(\alpha)>1,
\end{array}
\right.
\end{eqnarray}
where $\Phi(\lambda,\alpha,q)$ stands for the Lerch function. As for the asymptotic behavior of this function, when $z\to \infty,$ it is given by
\begin{equation}
R_K(z,-\alpha,q,\lambda)\sim\sqrt{\frac{\pi}{2z}}\frac{e^{-z\sqrt{q}}}{q^{\alpha/2+1/4}},\;\text{as }z\to \infty,\;|\arg z|<\frac{\pi}{4}, \; -1\leq\lambda\leq 1.
\end{equation}
In particular, when $q = 1$, we have
\begin{equation}
R_K(z,-\alpha,1,\lambda)\sim\sqrt{\frac{\pi}{2z}}{e^{-z}},\;\text{as }z\to \infty,\;|\arg z|<\frac{\pi}{4},
\end{equation}
which is the same asymptotic formula as that of  $K_{\alpha}$.

Further, by using the generalized extended beta function we get other extensions of  Gauss hypergeometric, confluent hypergeometric, Appell and Lauricella hypergeometric functions and we investigate some of their properties.\\

Recently, fractional derivative operators become significant research topics due to their wide applications in various areas including mathematical, physical, life sciences and engineering problems. To cite only a few of this operator's applications, we refer  to \cite{kilbas}, \cite{sun} and the references therein. The use of fractional derivative operators in obtaining generating relations for some special functions can be found in \cite{ozregin2010,Sriva}.
There are two important fractional derivatives operators: Riemann-Liouville and
Caputo operators.  Undoubtedly, the difference between them is very important for applications to differential equations because of required initial conditions which are of different types (see for further details \cite{Li} and \cite{zhou}). It is worth being pointed out that nowadays a great attention is devoted to develop extensions of fractional differential operators, readers may refer to \cite{agarwal,baleanu,Bohner,kiymaz,kiymaz1,ozregin2010,ozregin2011,parmar,rahman2018,Veling}. Making use of the $R_K$ function and
 inspired by the work of Agarwal et al. \cite{agarwal}, we introduce new generalized incomplete Riemann-Liouville fractional integral operators, and we obtain some generating relations involving generalized extended Gauss hypergeometric function.\\
 \par The paper is organized as follows: In section 2, we introduce the  generalized extended incomplete Gamma and Euler's beta functions, some of their properties are investigated. Section 3 is devoted to introduce extended hypergeometric and confulent hypergeometric functions by the extended Euler's beta function given in section 2, their related properties are established. The extended Appell and Lauricella hypergeometric function are given in section 4. In section 5, we give another result which consits to introduce the generalized extended Riemann Liouville fractional derivative operator and establish  most important properties such Mellin transform among others. Finally, in the last section, we obtain linear and bilinear generating relations for the generalized extended hypergeometric functions.

\section{The generalized  extended incomplete Gamma and Euler's beta functions}
In this section, we define new extended incomplete Gamma and Euler's beta functions based on the extension of  Bessel function \eqref{boujlkha} and we give some properties.
\subsection{The generalized extended incomplete Gamma function}
\begin{definition}
 The  generalized extended incomplete gamma functions  are given by
 \begin{eqnarray}
\gamma_\mu(\alpha, x;q;\lambda ; p)&=&\sqrt{\frac{2p}{\pi}}\int_{0}^{x}t^{\alpha-\frac{3}{2}}e^{-t}R_K\left(\frac{p}{t},-\mu-\frac{1}{2},q,\lambda\right)dt\label{gGamma1}\\
\Gamma_\mu(\alpha, x;q;\lambda ; p)&=&\sqrt{\frac{2p}{\pi}}\int_{x}^{\infty}t^{\alpha-\frac{3}{2}}e^{-t}R_K\left(\frac{p}{t},-\mu-\frac{1}{2},q,\lambda\right)dt\label{gGamma}
\end{eqnarray}
  where $\Re(x)>0$,  $0<q\leq 1$,  $-1\leq \lambda\leq 1$ and $\Re(p) > 0$.
\end{definition}
\begin{remark}
When $\lambda = 0$ and $q=1$, \eqref{gGamma1} and \eqref{gGamma} are respectively reduced to the extended incomplete gamma functions \eqref{chaudzubair1} and \eqref{chaudzubair2} defined by Chaudhry and Zubair \cite{Chaudhry1, Chaudhry2}.

\end{remark}

\begin{proposition}[Decomposition theorem]
\begin{eqnarray}\label{eq3}
\Gamma_\mu(\alpha, x;q;\lambda ; p)+\gamma_\mu(\alpha, x;q;\lambda ; p)&=&\frac{\Gamma(\alpha+\mu)}{\sqrt{\pi}}\left(\frac{p}{2}\right)^{-\mu}\Phi_{1-\frac{\alpha+\mu}{2},\frac{1}{2}-\frac{\alpha+\mu}{2}}\left(\lambda,\mu+\frac{1}{2},q,\frac{p^2}{16}\right)\nonumber\\
&&+\frac{\Gamma\left(-\frac{\alpha+\mu}{2}\right)}{2\sqrt{\pi}}\left(\frac{p}{2}\right)^{\alpha}\Phi_{\frac{1}{2},\frac{\alpha+\mu+2}{2}}\left(\lambda,\frac{\mu-\alpha+1}{2},q,\frac{p^2}{16}\right)\nonumber\\
&&-\frac{\Gamma\left(-\frac{\alpha+\mu+1}{2}\right)}{2\sqrt{\pi}}\left(\frac{p}{2}\right)^{\alpha+1}\Phi_{\frac{3}{2},\frac{\alpha+\mu+3}{2}}\left(\lambda,\frac{\mu-\alpha}{2},q,\frac{p^2}{16}\right),
\end{eqnarray}
with $\Re(p)>0,\,-\infty<\alpha<\infty$ and
\begin{eqnarray}\label{PHI}
\Phi_{b_1,b_2}(\lambda,s,q,\xi)&=&\int_{0}^{\infty}\frac{t^{s-1}e^{-qt}}{1-\lambda e^{-t}} \,{}_0F_2\left(
\begin{array}{ll}
-&\\
&;-\frac{\xi}{t}\\
b_1,b_2&
\end{array}
\right)
dt\nonumber\\
&=&\int_{0}^{\infty}\frac{t^{s-1}e^{-(q-1)t}}{e^{t}-\lambda } \,{}_0F_2\left(
\begin{array}{ll}
-&\\
&;-\frac{\xi}{t}\\
b_1,b_2&
\end{array}
\right)
dt,
\end{eqnarray}
  $s\in\mathbb{C},\,\Re(\xi)>0$ and $b_1,b_2\in \mathbb{C}\setminus\mathbb{Z}_0^{-}.$

\end{proposition}
\begin{proof}
We have
\begin{eqnarray}
\Gamma_\mu(\alpha, x;q;\lambda ; p)+\gamma_\mu(\alpha, x;q;\lambda ; p)&=&\sqrt{\frac{2p}{\pi}}\int_{0}^{\infty}t^{\alpha-\frac{3}{2}}e^{-t}R_K\left(\frac{p}{t},-\mu-\frac{1}{2},q,\lambda\right)dt\nonumber\\
&=&\frac{1}{\sqrt{\pi}}\left(\frac{p}{2}\right)^{-\mu}\int_{0}^{\infty}t^{\alpha+\mu-1}e^{-t} \left(\int_{0}^{\infty}\tau^{\mu-\frac{1}{2}}\frac{e^{-q\tau-\frac{p^2}{4t^2\tau}}}{1-\lambda e^{-\tau}}d\tau\right)dt\nonumber\\
&=&\frac{1}{\sqrt{\pi}}\left(\frac{p}{2}\right)^{-\mu} \int_{0}^{\infty}\tau^{\mu-\frac{1}{2}}\frac{e^{-q\tau}}{1-\lambda e^{-\tau}}\left(\int_{0}^{\infty}t^{\alpha+\mu-1}e^{-t}e^{-\frac{p^2}{4t^2\tau}}dt\right)d\tau.\label{int1}
\end{eqnarray}

 Using the integral \cite[pp. 31, formula 6]{prudnikov}, we obtain
 \begin{eqnarray}
 \int_{0}^{\infty}t^{\alpha+\mu-1}e^{-t}e^{-\frac{p^2}{4t^2\tau}}dt&=&\Gamma(\alpha+\mu)\,{}_0F_{2}\left(
\begin{array}{ll}
-&\\
&;-\frac{p^2}{16\tau}\\
1-\frac{\alpha+\mu}{2},\frac{1}{2}-\frac{\alpha+\mu}{2}&
\end{array}
\right)
\nonumber\\&&+\frac{\Gamma\left(-\frac{\alpha+\mu}{2}\right)}{2}\left(\frac{p^2}{4\tau}\right)^{\frac{\alpha+\mu}{2}}\,{}_0F_{2}\left(
\begin{array}{ll}
-&\\
&;-\frac{p^2}{16\tau}\\
\frac{1}{2},\frac{\alpha+\mu+2}{2}&
\end{array}
\right)\nonumber\\
&&-\frac{\Gamma\left(-\frac{\alpha+\mu+1}{2}\right)}{2}\left(\frac{p^2}{4\tau}\right)^{\frac{\alpha+\mu+1}{2}}\,{}_0F_{2}\left(
\begin{array}{ll}
-&\\
&;-\frac{p^2}{16\tau}\\
\frac{3}{2},\frac{\alpha+\mu+3}{2}&
\end{array}
\right).\label{int2}
 \end{eqnarray}

Finally, substituting \eqref{int2} in \eqref{int1} and by using the notation  \eqref{PHI}  we get the desired result.
\end{proof}

\begin{proposition}[Recurrence relation]

\begin{equation}
\Gamma_\mu(\alpha+1, x;q;\lambda ; p)=
(\alpha+\mu)\Gamma_\mu(\alpha, x;q;\lambda ; p)
+p\Gamma_{\mu-1}(\alpha-1, x;q;\lambda ; p)+\sqrt{\frac{2p}{\pi}}x^{\alpha-\frac{1}{2}}e^{-x}R_K\left(\frac{p}{x},-\mu-\frac{1}{2},q,\lambda\right),
\end{equation}
$(\Re(p)>0,\,-\infty<\alpha<\infty).$
\end{proposition}
\begin{proof}
We have
\begin{eqnarray}\label{dkt}
\frac{d}{dt}\left[ R_K\left(\frac{p}{t},-\mu-\frac{1}{2},q,\lambda\right)\right]&=&\frac{d}{dt}\left[ \frac{\left(\frac{p}{2t}\right)^{-\mu-\frac{1}{2}}}{2}\int_{0}^{\infty}\tau^{\mu-\frac{1}{2}}\frac{e^{-q\tau-\frac{p^2}{4t^2\tau}}}{1-\lambda e^{-\tau}}d\tau\right]\nonumber\\
&=&\frac{\mu+\frac{1}{2}}{t}R_K\left(\frac{p}{t},-\mu-\frac{1}{2},q,\lambda\right)+\frac{p}{t^2}R_K\left(\frac{p}{t},-\mu+\frac{1}{2},q,\lambda\right).
\end{eqnarray}

Differentiating  $t^{\alpha-\frac{1}{2}}e^{-t}R_K\left(\frac{p}{t},-\mu-\frac{1}{2},q,\lambda\right)$ with respect to $t$ and by using \eqref{dkt}, we get
\begin{eqnarray}\label{dkt2}
&&\frac{d}{dt}\left[t^{\alpha-\frac{1}{2}}e^{-t}R_K\left(\frac{p}{t},-\mu-\frac{1}{2},q,\lambda\right)\right]=
(\alpha+\mu)t^{\alpha-\frac{3}{2}}e^{-t}R_K\left(\frac{p}{t},-\mu-\frac{1}{2},q,\lambda\right)\nonumber\\
&&+p\,t^{\alpha-\frac{5}{2}}e^{-t}R_K\left(\frac{p}{t},-\mu+\frac{1}{2},q,\lambda\right)-t^{\alpha-\frac{1}{2}}e^{-t}R_K\left(\frac{p}{t},-\mu-\frac{1}{2},q,\lambda\right).
\end{eqnarray}

Multiplying both sides of \eqref{dkt2} by $\sqrt{\frac{2p}{\pi}}$ and integrating from $x$ to $\infty$ and using \eqref{gGamma}, we find
\begin{equation*}
0-\sqrt{\frac{2p}{\pi}}x^{\alpha-\frac{1}{2}}e^{-x}R_K\left(\frac{p}{x},-\mu-\frac{1}{2},q,\lambda\right)=
(\alpha+\mu)\Gamma_\mu(\alpha, x;q;\lambda ; p)
+p\Gamma_{\mu-1}(\alpha-1, x;q;\lambda ; p)-\Gamma_\mu(\alpha+1, x;q;\lambda ; p),
\end{equation*}
which can be also written as
\begin{equation*}
\Gamma_\mu(\alpha+1, x;q;\lambda ; p)=
(\alpha+\mu)\Gamma_\mu(\alpha, x;q;\lambda ; p)
+p\Gamma_{\mu-1}(\alpha-1, x;q;\lambda ; p)+\sqrt{\frac{2p}{\pi}}x^{\alpha-\frac{1}{2}}e^{-x}R_K\left(\frac{p}{x},-\mu-\frac{1}{2},q,\lambda\right).
\end{equation*}
\end{proof}

\begin{proposition} The following formula holds
\begin{equation}
\Gamma_{\mu-1}(\alpha, x;1;\lambda ; p)-\Gamma_{\mu+1}(\alpha, x;1;\lambda ; p)+\frac{2\mu+1}{p}\Gamma_{\mu}(\alpha+1, x;1;\lambda ; p)=\lambda\frac{\partial}{\partial \lambda}\Gamma_{\mu+1}(\alpha, x;1;\lambda ; p),
\end{equation}
$(\Re(p)>0,\,-\infty<\alpha<\infty).$
\end{proposition}
\begin{proof}
By using \eqref{gGamma}, for $q=1$ and the following relation \cite[(22)]{Boudjelkha}, we get
\begin{equation}
R_K(z,-\alpha+1,1,\lambda)-R_K(z,-\alpha-1,1,\lambda)+\frac{2\alpha}{z}R_K(z,-\alpha,1,\lambda)=\lambda\frac{\partial}{\partial \lambda}R_K(z,-\alpha-1,1,\lambda).
\end{equation}
\end{proof}

\begin{proposition}[Laplace transform]
Let\begin{equation*}
H(\tau)=\left\{
\begin{array}{ll}
1&\tau>0\\
0&\tau<0
\end{array}\right.
\end{equation*}
be the Heaviside unit step function and ${\cal{L}}$ be the Laplace transform operator. Then
\begin{equation}\label{lap1}
{\cal{L}}\left\{t^{\alpha-\frac{3}{2}}R_K\left(\frac{p}{t},-\mu-\frac{1}{2},q,\lambda\right)H(t-x);s\right\}=\sqrt{\frac{\pi}{2p}}s^{-\alpha}\Gamma_\mu(\alpha, sx;q;\lambda ; sp),
\end{equation}
\begin{equation}\label{lap2}
{\cal{L}}\left\{t^{\alpha-\frac{3}{2}}R_K\left(\frac{p}{t},-\mu-\frac{1}{2},q,\lambda\right)H(t-x)H(t);s\right\}=\sqrt{\frac{\pi}{2p}}s^{-\alpha}\gamma_\mu(\alpha, sx;q;\lambda ; sp),
\end{equation}
($x>0$, $\Re (p)>0$, $-\infty<\alpha<\infty$).
\end{proposition}
\begin{proof}
We have
\begin{eqnarray*}
{\cal{L}}\left\{t^{\alpha-\frac{3}{2}}R_K\left(\frac{p}{t},-\mu-\frac{1}{2},q,\lambda\right)H(t-x);s\right\}&=&\int_{0}^{\infty}t^{\alpha-\frac{3}{2}}R_K\left(\frac{p}{t},-\mu-\frac{1}{2},q,\lambda\right)e^{-st}H(t-x)dt\nonumber\\
&=&\int_{x}^{\infty}t^{\alpha-\frac{3}{2}}R_K\left(\frac{p}{t},-\mu-\frac{1}{2},q,\lambda\right)e^{-st}dt.
\end{eqnarray*}

Substituting $t=\frac{\tau}{s}$, $dt=\frac{d\tau}{s}$, we get
\begin{eqnarray*}
\int_{x}^{\infty}t^{\alpha-\frac{3}{2}}R_K\left(\frac{p}{t},-\mu-\frac{1}{2},q,\lambda\right)e^{-st}dt&=&s^{-\alpha+\frac{1}{2}}\int_{sx}^{\infty}\tau^{\alpha-\frac{3}{2}}e^{-\tau}R_K\left(\frac{sp}{\tau},-\mu-\frac{1}{2},q,\lambda\right)dt\\
&=&\sqrt{\frac{\pi}{2p}}s^{-\alpha}\Gamma_\mu(\alpha, sx;q;\lambda ; sp).
\end{eqnarray*}

The proof of \eqref{lap2} is omitted since it is quite similar as that of \eqref{lap1}.
\end{proof}

\begin{proposition}[Parametric differentiation]

\begin{equation}
\frac{\partial}{\partial p}\left(\Gamma_\mu(\alpha, x;q;\lambda ; p)\right)=
-\frac{1}{p}\left[\mu\Gamma_\mu(\alpha, x;q;\lambda ; p)
+p\Gamma_{\mu-1}(\alpha-1, x;q;\lambda ; p)\right].
\end{equation}
\end{proposition}
\begin{proof}
\begin{eqnarray}\label{par1}
\frac{\partial}{\partial p}\left(\Gamma_\mu(\alpha, x;q;\lambda ; p)\right)&=&\frac{1}{2p}\sqrt{\frac{2p}{\pi}}\int_{x}^{\infty}t^{\alpha-\frac{3}{2}}e^{-t}R_K\left(\frac{p}{t},-\mu-\frac{1}{2},q,\lambda\right)dt\nonumber\\&&+\sqrt{\frac{2p}{\pi}}\int_{x}^{\infty}t^{\alpha-\frac{3}{2}}e^{-t}\frac{\partial}{\partial p}\left(R_K\left(\frac{p}{t},-\mu-\frac{1}{2},q,\lambda\right)\right)dt.
\end{eqnarray}

We have
\begin{eqnarray}\label{par2}
\frac{\partial}{\partial p}\left(R_K\left(\frac{p}{t},-\mu-\frac{1}{2},q,\lambda\right)\right)&=&-\frac{\mu+\frac{1}{2}}{p}\frac{(p/2t)^{-\mu-\frac{1}{2}}}{2}\int_{0}^{\infty}\tau^{\mu-\frac{1}{2}}\frac{e^{-q\tau- \frac{p^2}{4t^2\tau} }}{1-\lambda e^{-\tau}}d\tau-\frac{1}{t}\frac{(p/2t)^{-\mu+\frac{1}{2}}}{2}\int_{0}^{\infty}\tau^{\mu-\frac{3}{2}}\frac{e^{-q\tau- \frac{p^2}{4t^2\tau} }}{1-\lambda e^{-\tau}}d\tau\nonumber\\
&=&-\frac{\mu+\frac{1}{2}}{p}R_K\left(\frac{p}{t},-\mu-\frac{1}{2},q,\lambda\right)-\frac{1}{t}R_K\left(\frac{p}{t},-\mu+\frac{1}{2},q,\lambda\right),
\end{eqnarray}

Finally, by Substituting \eqref{par2} into \eqref{par1} we get the desired result.
\end{proof}

\subsection{The  generalized extended beta  function}

\begin{definition}
 The generalized  extended beta function  is given by

\begin{equation}\label{gbeta}
B_\mu(x, y;q;\lambda; p;m)=\sqrt{\frac{2p}{\pi}}\int_{0}^{1}t^{x-\frac{3}{2}}(1-t)^{y-\frac{3}{2}}R_K\left(\frac{p}{t^m(1-t)^m},-\mu-\frac{1}{2},q,\lambda\right)dt,
\end{equation}
where $x, y \in\mathbb{C}$, $0<q\leq 1$,  $-1\leq \lambda\leq 1$, $m > 0$ and $\Re(p) > 0$.
\end{definition}
\begin{remark}
Taking $\lambda = 0$ and $q=1$, \eqref{gbeta} is reduced to the extended   Euler's beta function \eqref{agarwal} defined by Agarwal et al. \cite{agarwal}.

\end{remark}

\begin{proposition}[Functional relations]\text{}

1. The following formula holds
 \begin{equation}\label{beta1}
B_\mu(x, y;q;\lambda; p;m)=B_\mu(x+1, y;q;\lambda; p;m)+B_\mu(x, y+1;q;\lambda; p;m).
\end{equation}

2. Let $n\in\mathbb{N}$. Then, the following summation formula holds
\begin{equation}\label{beta2}
B_\mu(x, y;q;\lambda; p;m)=\sum_{k=0}^nB_\mu(x+k, y+n-k;q;\lambda; p;m).
\end{equation}
\end{proposition}

\begin{proof}

1. The right-hand side of \eqref{beta1} yields to
\begin{equation*}
\sqrt{\frac{2p}{\pi}}\int_{0}^{1}\left\{t^{x-\frac{1}{2}}(1-t)^{y-\frac{3}{2}}+t^{x-\frac{3}{2}}(1-t)^{y-\frac{1}{2}}\right\}
R_K\left(\frac{p}{t^m(1-t)^m},-\mu-\frac{1}{2},q,\lambda\right)dt,
\end{equation*}
which, after simplification, implies
\begin{equation*}
\sqrt{\frac{2p}{\pi}}\int_{0}^{1}t^{x-\frac{3}{2}}(1-t)^{y-\frac{3}{2}}R_K\left(\frac{p}{t^m(1-t)^m},-\mu-\frac{1}{2},q,\lambda\right)dt,
\end{equation*}
which is equal to the left-hand side of \eqref{beta1}.

2. The case $n=0$ of \eqref{beta2} holds easily. The case $n=1$ of \eqref{beta2} is just  \eqref{beta1}. For the other cases we can easily proceed by induction on  $n.$
\end{proof}

\begin{proposition}\label{propbeta}
The following formula holds
\begin{equation}\label{beta3}
B_\mu(x, 1-y;q;\lambda; p;m)=\sum_{n=0}^\infty \frac{(y)_n}{n!}B_\mu(x+n, 1;q;\lambda; p;m).
\end{equation}

\end{proposition}

\begin{proof}
We have
\begin{equation}\label{beta4}
B_\mu(x, 1-y;q;\lambda; p;m)=\sqrt{\frac{2p}{\pi}}\int_{0}^{1}t^{x-\frac{3}{2}}(1-t)^{-y-\frac{1}{2}}R_K\left(\frac{p}{t^m(1-t)^m},-\mu-\frac{1}{2},q,\lambda\right)dt.
\end{equation}

By substituting  the formula
\begin{eqnarray}
(1-t)^{-y}=\sum_{n=0}^{\infty}(y)_n\frac{t^n}{n!},\;\;(|t|<1,\;\;y\in\mathbb{C}),
\end{eqnarray}
in the right-hand of \eqref{beta4} and after interchanging the order of integral and summation, we get \eqref{beta3}.
\end{proof}

\begin{proposition}
The following formula holds
\begin{equation}\label{beta3}
B_\mu(x, y;q;\lambda; p;m)=\sum_{n=0}^\infty B_\mu(x+n, y+1;q;\lambda; p;m).
\end{equation}

\end{proposition}

\begin{proof}
By substituting again the formula
\begin{eqnarray*}
(1-t)^{y-1}=(1-t)^y\sum_{n=0}^{\infty}t^n,\;\;(|t|<1),
\end{eqnarray*}
in the right-hand of \eqref{gbeta} and similarly as in the proof of Proposition \ref{propbeta} we get the desired result.
\end{proof}

\begin{lemma}\label{lema1}
Let ${\cal{M}}$ be the Mellin transform operator. Then
\begin{equation}
{\cal{M}}\{R_K(z,-\alpha,q,\lambda),z\to s\}=2^{s-2}\Gamma\left(\frac{s-\alpha}{2}\right)\Gamma\left(\frac{s+\alpha}{2}\right)\Phi\left(\lambda,\frac{s+\alpha}{2},q\right),
\end{equation}
where $0<q\leq 1$, or $-1\leq \lambda<1$, $\Re(s)>|\Re(\alpha)|$ or $\lambda=1$, $\Re(s)>\max(\Re(\alpha),2-\Re(\alpha))$ and $\Phi\left(\lambda,\frac{s+\alpha}{2},q\right)$ stands for the Lerch function (see \cite{gradshteyn},\cite{Fraczek})
\end{lemma}
\begin{proof}
\begin{eqnarray*}
{\cal{M}}\{R_K(z,-\alpha,q,\lambda),z\to s\}&=&\int_{0}^{\infty}z^{s-1}R_K(z,-\alpha,q,\lambda)dz=2^{\alpha-1}\int_{0}^{\infty}z^{s-\alpha-1}\left(\int_{0}^{\infty}t^{\alpha-1}\frac{e^{-qt-z^2/4t}}{1-\lambda e^{-t}}dt\right)dz\nonumber\\
&=&2^{\alpha-1}\int_{0}^{\infty}t^{\alpha-1}\frac{e^{-qt}}{1-\lambda e^{-t}}\left(\int_{0}^{\infty}z^{s-\alpha-1}e^{-z^2/4t}dz\right)dt\nonumber\\
&=&2^{s-2}\Gamma\left(\frac{s-\alpha}{2}\right)\int_{0}^{\infty}t^{\frac{s+\alpha}{2}-1}\frac{e^{-qt}}{1-\lambda e^{-t}}dt\nonumber\\
&=&2^{s-2}\Gamma\left(\frac{s-\alpha}{2}\right)\Gamma\left(\frac{s+\alpha}{2}\right)\Phi\left(\lambda,\frac{s+\alpha}{2},q\right).
\end{eqnarray*}
\end{proof}

\begin{proposition}[Mellin transform]\label{proMelin} The following expression holds true
\begin{equation}
{\cal{M}}\{B_\mu(x, y;q;\lambda; p;m),p\to s\}=\frac{2^{s-1}}{\sqrt{\pi}}B\left(x+ms+\frac{m-1}{2},y+ms+\frac{m-1}{2}\right)\Gamma\left(\frac{s-\mu}{2}\right)\Gamma\left(\frac{s+\mu+1}{2}\right)\Phi\left(\lambda,\frac{s+\mu+1}{2},q\right),
\end{equation}
where $x,y\in\mathbb{C}$, $m>0$ and
$$0<q\leq1,\,or\,-1\leq\lambda<1,\,\Re(s)>\max\left\{\Re(\mu),-1-\Re(\mu),-\frac{1}{2}+\frac{1}{2m}-\frac{\Re(x)}{m},-\frac{1}{2}+\frac{1}{2m}-\frac{\Re(y)}{m}\right\},\,$$ $$or\,\lambda=1,\,\Re(s)>\max\left\{\Re(\mu),1-\Re(\mu),-\frac{1}{2}+\frac{1}{2m}-\frac{\Re(x)}{m},-\frac{1}{2}+\frac{1}{2m}-\frac{\Re(y)}{m}\right\}.$$
\end{proposition}
\begin{proof}
\begin{eqnarray*}
{\cal{M}}\{B_\mu(x, y;q;\lambda; p;m),p\to s\}&=&\int_{0}^{\infty}p^{s-1}B_\mu(x, y;q;\lambda; p;m)dp\nonumber\\
&=&\int_{0}^{\infty}p^{s-1}\sqrt{\frac{2p}{\pi}}\left(\int_{0}^{1}t^{x-\frac{3}{2}}(1-t)^{y-\frac{3}{2}}R_K\left(\frac{p}{t^m(1-t)^m},-\mu-\frac{1}{2},q,\lambda\right)dt\right)dp\nonumber\\
&=&\sqrt{\frac{2}{\pi}}\int_{0}^{1}t^{x-\frac{3}{2}}(1-t)^{y-\frac{3}{2}}\left(\int_{0}^{\infty}p^{s+\frac{1}{2}-1}R_K\left(\frac{p}{t^m(1-t)^m},-\mu-\frac{1}{2},q,\lambda\right)dp\right)dt\nonumber\\
&=&\sqrt{\frac{2}{\pi}}\int_{0}^{1}t^{x+m(s+\frac{1}{2})-\frac{3}{2}}(1-t)^{y+m(s+\frac{1}{2})-\frac{3}{2}}dt\int_{0}^{\infty}u^{s+\frac{1}{2}-1}R_K\left(u,-\mu-\frac{1}{2},q,\lambda\right)du\nonumber\\
&=&\sqrt{\frac{2}{\pi}}B\left(x+ms+\frac{m-1}{2},y+ms+\frac{m-1}{2}\right)\int_{0}^{\infty}u^{s+\frac{1}{2}-1}R_K\left(u,-\mu-\frac{1}{2},q,\lambda\right)du.
\end{eqnarray*}

Finally, by using  Lemma \ref{lema1} we get the desired result.
\end{proof}

\section{Extended Gauss hypergeometric and confluent hypergeometric  functions}

We use the generalized  extended beta function \eqref{gbeta} to extend hypergeometric and confluent hypergeometric functions, respectively, as follows:

\begin{definition}\label{Defin}
The extended Gauss hypergeometric function $F_\mu(a,b;c;z;q;\lambda; p;m)$ and the confluent hypergeometric function  $\Phi_\mu(b;c;z;q;\lambda; p;m)$ are respectively defined by
\begin{eqnarray}\label{gbeta2}
F_\mu(a,b;c;z;q;\lambda; p;m)=\sum_{n=0}^{\infty}(a)_n\frac{B_\mu(b+n, c-b;q;\lambda; p;m)}{B(b, c-b)}\frac{z^n}{n!}
\end{eqnarray}
($|z|<1$, $\Re(c)>\Re(b)>0$, $0<q\leq 1$,  $-1\leq \lambda\leq 1$, $m > 0$, $\Re(p) > 0$).
\begin{eqnarray}
\Phi_\mu(b;c;z;q;\lambda; p;m)=\sum_{n=0}^{\infty}\frac{B_\mu(b+n, c-b;q;\lambda; p;m)}{B(b, c-b)}\frac{z^n}{n!}
\end{eqnarray}
($z\in\mathbb{C}$, $\Re(c)>\Re(b)>0$,  $-1\leq \lambda\leq 1$, $m > 0$, $\Re(p) > 0$).
\end{definition}
\begin{remark}
Taking $\lambda = 0$ and $q=1$, \eqref{gbeta2} reduces to the  extended Gauss hypergeometric function  defined by Agarwal et al. \cite[Definition 2.8]{agarwal}.

\end{remark}

\begin{proposition}[Integral representation]
1. The following integral representation for the extended Gauss hypergeometric function $F_\mu(a,b;c;z;q;\lambda; p;m)$ is valid
\begin{equation}\label{intrep1}
F_\mu(a,b;c;z;q;\lambda; p;m)=\sqrt{\frac{2p}{\pi}}\frac{1}{B(b,c-b)}\int_0^1 t^{b-\frac{3}{2}}(1-t)^{c-b-\frac{3}{2}}(1-zt)^{-a}R_K\left(\frac{p}{t^m(1-t)^m},-\mu-\frac{1}{2},q,\lambda\right)dt,
\end{equation}
($\arg(1-z)<\pi$, $\Re(c)>\Re(b)>0$, $0<q\leq 1$,  $-1\leq \lambda\leq 1$, $m > 0$, $\Re(p) > 0$).

2. The following integral representation for the extended confluent hypergeometric function $\Phi_\mu(b;c;z;q;\lambda; p;m)$ is valid
\begin{equation}\label{intrep2}
\Phi_\mu(b;c;z;q;\lambda; p;m)=\sqrt{\frac{2p}{\pi}}\frac{1}{B(b,c-b)}\int_0^1 t^{b-\frac{3}{2}}(1-t)^{c-b-\frac{3}{2}}e^{zt}R_K\left(\frac{p}{t^m(1-t)^m},-\mu-\frac{1}{2},q,\lambda\right)dt,
\end{equation}
( $\Re(c)>\Re(b)>0$, $0<q\leq 1$,  $-1\leq \lambda\leq 1$, $m > 0$, $\Re(p) > 0$).

\end{proposition}
\begin{proof}
1. By using \eqref{gbeta} and the generalized binomial expansion
\begin{equation}\label{gbeta1}
(1-zt)^{-a}=\sum_{n=0}^{\infty}(a)_n\frac{(zt)^n}{n!},\;\;(|zt|<1),
\end{equation}
we get the required result.

2. Similarly as in the proof of 1.
\end{proof}
\begin{proposition}[Differentiation formula]
(a) For $n\in \mathbb{N}$,
\begin{equation}\label{difform1}
\frac{d^n}{dz^n}\{F_\mu(a,b;c;z;q;\lambda; p;m)\}=\frac{(a)_n(b)_n}{(c)_n}F_\mu(a+n,b+n;c+n;z;q;\lambda; p;m),
\end{equation}
($|z|<1$, $\Re(c)>\Re(b)>0$, $0<q\leq 1$,  $-1\leq \lambda\leq 1$, $m > 0$, $\Re(p) > 0$).

(b) For $n\in \mathbb{N}$,
\begin{equation}
\frac{d^n}{dz^n}\{\Phi_\mu(b;c;z;q;\lambda; p;m)\}=\frac{(b)_n}{(c)_n}\Phi_\mu(b+n;c+n;z;q;\lambda; p;m),
\end{equation}
($z\in\mathbb{C}$, $\Re(c)>\Re(b)>0$, $0<q\leq 1$,  $-1\leq \lambda\leq 1$, $m > 0$, $\Re(p) > 0$).
\end{proposition}

\begin{proof}
(a) For $n=1$, we have
\begin{eqnarray}
\frac{d}{dz}\{F_\mu(a,b;c;z;q;\lambda; p;m)\}&=&\sum_{n=1}^{\infty}(a)_n\frac{B_\mu(b+n, c-b;q;\lambda; p;m)}{B(b, c-b)}\frac{z^{n-1}}{(n-1)!}\nonumber\\
&=&\sum_{n=0}^{\infty}(a)_{n+1}\frac{B_\mu(b+n+1, c-b;q;\lambda; p;m)}{B(b, c-b)}\frac{z^{n}}{n!}.\label{dif}
\end{eqnarray}

Using identities $B(b,c-b)=\frac{c}{b}B(b+1,c-b)$ and $(a)_{n+1}=a(a+1)_n$ in \eqref{dif}, we get
\begin{eqnarray}
\frac{d}{dz}\{F_\mu(a,b;c;z;q;\lambda; p;m)\}&=&\frac{ab}{c}\sum_{n=0}^{\infty}(a+1)_n\frac{B_\mu(b+n+1, c-b;q;\lambda; p;m)}{B(b+1, c-b)}\frac{z^{n}}{n!}\nonumber\\
&=&\frac{ab}{c}F_\mu(a+1,b+1;c+1;z;q;\lambda; p;m),
\end{eqnarray}
and hence
\begin{equation}\label{dif2}
\frac{d}{dz}\{F_\mu(a,b;c;z;q;\lambda; p;m)\}=\frac{ab}{c}F_\mu(a+1,b+1;c+1;z;q;\lambda; p;m).
\end{equation}

Then, by using \eqref{dif2} repeatedly, we get \eqref{difform1}.

The proof of part (b) is similar as that of  part (a).
\end{proof}

\begin{proposition}[Transformation formulas]\text{}

1. For $\arg(1-z)<\pi$, we have
\begin{equation}
F_\mu(a,b;c;z;q;\lambda; p;m)=(1-z)^{-a}F_\mu(a,c-b;c;\frac{z}{z-1};q;\lambda; p;m),
\end{equation}
( $\Re(c)>\Re(b)>0$, $0<q\leq 1$,  $-1\leq \lambda\leq 1$, $m > 0$, $\Re(p) > 0$).

2. \begin{equation}
\Phi_\mu(b;c;z;q;\lambda; p;m)=e^{z}\Phi_\mu(c-b;c;-z;q;\lambda; p;m),
\end{equation}
($z\in\mathbb{C}$, $\Re(c)>\Re(b)>0$, $0<q\leq 1$,  $-1\leq \lambda\leq 1$, $m > 0$, $\Re(p) > 0$).

\end{proposition}
\begin{proof}
 Replacing $t$ by $1-t$ in the integral representations  \eqref{intrep1} and \eqref{intrep2}.
\end{proof}

\section{Extended Appell and Lauricella hypergeometric   functions}

\begin{definition}
Extended Appell hypergeometric functions $F_{1,\mu}$, $F_{2,\mu}$  and the Lauricella hypergeometric function $F_{D,\mu}^3$ are, respectively, defined by
\begin{equation}\label{gbeta3}
F_{1,\mu}(a,b,c;d;x,y;q;\lambda; p;m)=\sum_{n,k=0}^{\infty}(b)_n(c)_k\frac{B_\mu(a+n+k, d-a;q;\lambda; p;m)}{B(a, d-a)}\frac{x^n}{n!}\frac{y^k}{k!},
\end{equation}
($|x|<1$, $|y|<1$, $\Re(d)>\Re(a)>0$, $0<q\leq 1$,  $-1\leq \lambda\leq 1$, $m > 0$, $\Re(p) > 0$).

\begin{equation}\label{gbeta5}
F_{2,\mu}(a,b,c;d,e;x,y;q;\lambda; p;m)=\sum_{n,k=0}^{\infty}(a)_{n+k}\frac{B_\mu(b+n, d-b;q;\lambda; p;m)}{B(b, d-b)}\frac{B_\mu(c+k, e-c;q;\lambda; p;m)}{B(c, e-c)}\frac{x^n}{n!}\frac{y^k}{k!},
\end{equation}
($|x|+|y|<1$, $\Re(d)>\Re(b)>0$,  $\Re(e)>\Re(c)>0$, $0<q\leq 1$,  $-1\leq \lambda\leq 1$, $m > 0$, $\Re(p) > 0$).

\begin{equation}\label{gbeta6}
F_{D,\mu}^3(a,b,c,d;e;x,y,z;q;\lambda; p;m)=\sum_{n,k,r=0}^{\infty}(b)_n(c)_k(d)_r\frac{B_\mu(a+n+k+r, e-a;q;\lambda; p;m)}{B(a, e-a)}\frac{x^n}{n!}\frac{y^k}{k!}\frac{z^r}{r!},
\end{equation}
($|x|<1$, $|y|<1$, $|z|<1$, $\Re(e)>\Re(a)>0$, $0<q\leq 1$,  $-1\leq \lambda\leq 1$, $m > 0$, $\Re(p) > 0$).
\end{definition}
\begin{remark}
Taking $\lambda = 0$ and $q=1$, \eqref{gbeta3}, \eqref{gbeta5} and \eqref{gbeta6} are reduced to extended Appell hypergeometric functions $F_{1,\mu}$, $F_{2,\mu}$  and the Lauricella hypergeometric function $F_{D,\mu}^3,$  defined by Agarwal et al. \cite[Definitions 2.9, 2.10,2.11]{agarwal}.
\end{remark}

\begin{proposition}[Integral representation]
 The following integral representations for the extended Appell hypergeometric functions $F_{1,\mu}$, $F_{2,\mu}$  and the Lauricella hypergeometric function $F_{D,\mu}^3$ are, respectively, valid
\begin{equation}\label{intrep11}
F_{1,\mu}(a,b,c;d;x,y;q;\lambda; p;m)=\sqrt{\frac{2p}{\pi}}\frac{1}{B(a,d-a)}\int_0^1 t^{a-\frac{3}{2}}(1-t)^{d-a-\frac{3}{2}}(1-xt)^{-b}(1-yt)^{-c}R_K\left(\frac{p}{t^m(1-t)^m},-\mu-\frac{1}{2},q,\lambda\right)dt,
\end{equation}

\begin{eqnarray}\label{intrep12}
F_{2,\mu}(a,b,c;d;x,y;q;\lambda; p;m)=\frac{2p}{\pi}\frac{1}{B(b,d-b)B(c,e-c)}\int_0^1 \int_{0}^{1}t^{b-\frac{3}{2}}(1-t)^{d-b-\frac{3}{2}}w^{b-\frac{3}{2}}(1-w)^{e-c-\frac{3}{2}}(1-xt-yw)^{-a}\nonumber\\
\times R_K\left(\frac{p}{t^m(1-t)^m},-\mu-\frac{1}{2},q,\lambda\right)R_K\left(\frac{p}{w^m(1-w)^m},-\mu-\frac{1}{2},q,\lambda\right)dtdw,
\end{eqnarray}

\begin{eqnarray}\label{intrep13}
F_{D,\mu}^3(a,b,c,d;e;x,y,z;q;\lambda; p;m)=\sqrt{\frac{2p}{\pi}}\frac{1}{B(a,e-a)}\int_0^1 t^{a-\frac{3}{2}}(1-t)^{e-a-\frac{3}{2}}(1-xt)^{-b}(1-yt)^{-c}(1-zt)^{-d}\nonumber\\\times R_K\left(\frac{p}{t^m(1-t)^m},-\mu-\frac{1}{2},q,\lambda\right)dt.
\end{eqnarray}

\end{proposition}
\begin{proof}
The proofs are very similar to those of Theorems 2.13,  2.15 and  2.16 in \cite{agarwal}.
\end{proof}

\section{The generalized extended  Riemann-Liouville fractional derivative operator}
The classical Riemann-Liouville fractional derivative operator is defined by
\begin{equation}\label{classical}
    D_{z}^{\delta}f(z):=\frac{1}{\Gamma(-\delta)}\int_{0}^{z}(z-t)^{-\delta-1}f(t)dt,
\end{equation}
where $\Re(\delta)<0$. It coincides with the fractional integral of order $-\delta$. In the case $n-1<\Re(\delta)<n$, $n\in\mathbb{N}$, we write
\begin{equation}
    D_{z}^{\delta}f(z):=\frac{d^{n}}{dz^{n}}D_{z}^{\delta-n}f(z)=\frac{d^{n}}{dz^{n}}\left\{\frac{1}{\Gamma(n-\delta)}\int_{0}^{z}(z-t)^{n-\delta-1}f(t)dt\right\}.
\end{equation}

\begin{definition}\label{definito}
The  generalized extended Riemann-Liouville fractional derivative is defined as follows
\begin{equation}\label{definition1}
    D_{z}^{\delta,\mu;p;q;\lambda;m}f(z):=\frac{1}{\Gamma(-\delta)}\sqrt{\frac{2p}{\pi}}\int_{0}^{z}(z-t)^{-\delta-1}
    f(t)R_{K}\left(\frac{pz^{2m}}{t^{m}(z-t)^{m}},-\mu-\frac{1}{2},q,\lambda\right)dt,
\end{equation}
where $\Re(\delta)<0$, $\Re(p)>0$, $\Re(m)>0$, $\Re(\mu)\geq 0$ and $0<q\leq1$, $-1\leq\lambda\leq1$.

For $n-1<\Re(\delta)<n$, $n\in\mathbb{N},$ we have
\begin{equation}\label{definition12}
  D_{z}^{\delta,\mu;p;q;\lambda;m}f(z):=\frac{d^{n}}{dz^{n}}D_{z}^{\delta-n,\mu;p;q;\lambda;m}f(z)=\frac{d^{n}}{dz^{n}}\left\{\frac{1}{\Gamma(n-\delta)}\sqrt{\frac{2p}{\pi}}\int_{0}^{z}(z-t)^{n-\delta-1}f(t)R_{K}\left(\frac{pz^{2m}}{t^{m}(z-t)^{m}},-\mu-\frac{1}{2},q,\lambda\right)dt\right\}.
\end{equation}
\end{definition}
\begin{remark}

1. Taking $\lambda = 0$ and $q=1$, the generalized extended Riemann-Liouville fractional derivative operator \eqref{definition1} is reduced to the extended Riemann-Liouville fractional derivative operator given by Agarwal et al. \cite{agarwal}
\begin{equation}\label{definitionAgrwal}
    D_{z}^{\delta,\mu;p;m}f(z):=\frac{1}{\Gamma(-\delta)}\sqrt{\frac{2p}{\pi}}\int_{0}^{z}(z-t)^{-\delta-1}f(t)
    K_{\mu+\frac{1}{2}}\left(\frac{pz^{2m}}{t^{m}(z-t)^{m}}\right)dt,
\end{equation}
where $\Re(\delta)<0$, $\Re(p)>0$, $\Re(m)>0$, $\Re(\mu)>0$.

2. If $\lambda = 0$, $q=1$, $m=0$, $\mu=0$ and $p\to 0,$ then  the generalized extended Riemann-Liouville fractional
derivative operator \eqref{definition1} reduces to the classical Riemann-Liouville fractional derivative operator \eqref{classical}.
\end{remark}
\par In order to calculate  generalized extended fractional derivatives for some functions, we give  two results concerning the generalized extended Riemann-Liouville fractional derivative operator of some elementary functions which will be useful in the sequel.
\begin{lemma}\label{lem1}Let $\Re(\delta)<0$. Then, we have
\begin{equation}\label{lemma1}
     D_{z}^{\delta,\mu;p;q;\lambda;m}\{z^{\beta}\}=\frac{z^{\beta-\delta}}{\Gamma(-\delta)}B_{\mu}(\beta+\frac{3}{2},-\delta+\frac{1}{2};p;q;\lambda;m).
\end{equation}
\end{lemma}
\begin{proof} Using Definition \ref{definito}, and a local setting $t=zu,$ we obtain
\begin{eqnarray*}
    D_{z}^{\delta,\mu;p;q;\lambda;m}\{z^{\beta}\}&=& \frac{1}{\Gamma(-\delta)}\sqrt{\frac{2p}{\pi}}\int_{0}^{z}(z-t)^{-\delta-1}t^{\beta}R_{K}\left(\frac{pz^{2m}}{t^{m}(z-t)^{m}},-\mu-\frac{1}{2},q,\lambda\right)dt\\
&=& \frac{z^{\beta-\delta}}{\Gamma(-\delta)}\sqrt{\frac{2p}{\pi}}\int_{0}^{1}(1-u)^{(-\delta+\frac{1}{2})-\frac{3}{2}}u^{(\beta+\frac{3}{2})-\frac{3}{2}}R_{K}\left(\frac{p}{u^{m}(1-u)^{m}},-\mu-\frac{1}{2},q,\lambda\right)du \\
 &=&\frac{z^{\beta-\delta}}{\Gamma(-\delta)}B_{\mu}(\beta+\frac{3}{2},-\delta+\frac{1}{2};p;q;\lambda;m).
\end{eqnarray*}
\end{proof}
More generally, we give the  generalized extended Riemann-Liouville fractional derivative of an analytic function $f(z)$  at the origin.
\begin{lemma}\label{lem2} Let $\Re(\delta)<0$. If a function $f(z)$ is analytic at the origin, then
\begin{equation*}
    D_{z}^{\delta,\mu;p;q;\lambda;m}\{f(z)\}=\sum_{n=0}^{\infty}a_{n} D_{z}^{\delta,\mu;p;q;\lambda;m}\{z^{n}\}.
\end{equation*}
\end{lemma}

\begin{proof} Since $f$ is analytic at the origin, its Maclaurin expansion is given by $f(z)=\sum_{n=0}^{\infty}a_{n}z^{n}$ (for $|z|<\rho$ with $\rho\in\mathbb{R}^{+}$ is the convergence radius). By substituting entire power series in Definition \ref{definito}, we obtain
\begin{equation*}
  D_{z}^{\delta,\mu;p;q;\lambda;m}\{f(z)\}=\frac{1}{\Gamma(-\delta)}\sqrt{\frac{2p}{\pi}}\int_{0}^{z}(z-t)^{-\delta-1}R_{K}\left(\frac{pz^{2m}}{t^{m}(z-t)^{m}},-\mu-\frac{1}{2};q;\lambda\right)\sum_{n=0}^{\infty}a_{n}t^{n}dt.
\end{equation*}
By virtue of the uniform continuity on the convergence disk, we can do integration term by term in the equation above. Thus
\begin{eqnarray*}
   D_{z}^{\delta,\mu;p;q;\lambda;m}\{f(z)\} &=& \sum_{n=0}^{\infty}a_{n} \left\{\frac{1}{\Gamma(-\delta)}\sqrt{\frac{2p}{\pi}}\int_{0}^{z}(z-t)^{-\delta-1}R_{K}\left(\frac{pz^{2m}}{t^{m}(z-t)^{m}},-\mu-\frac{1}{2};q;\lambda\right)t^{n}dt\right\} \\
   &=& \sum_{n=0}^{\infty}a_{n} D_{z}^{\delta,\mu;p;q;\lambda;m}\{z^{n}\}.
\end{eqnarray*}
\end{proof}
\begin{corollary}
\begin{equation*}
     D_{z}^{\delta,\mu;p;q;\lambda;m}\{(1-z)^{-\alpha}\}=\frac{z^{-\delta}}{\Gamma(-\delta)}B\left(\frac{3}{2},
     -\delta+\frac{1}{2}\right)F_{\mu}(\alpha,\frac{3}{2},-\delta+2;z;q;\lambda;p;m),
\end{equation*}
where $\Re(\alpha)>0$ and $\Re(\delta)<0$.
\end{corollary}

\begin{proof}
Using binomial Theorem for $(1-z)^{-\alpha}$ and Lemma \ref{lem1}, we obtain:
\begin{eqnarray*}
  D_{z}^{\delta,\mu;p;q;\lambda;m}\{(1-z)^{-\alpha}\}&=&D_{z}^{\delta,\mu;p;q;\lambda;m}\left\lbrace \sum_{n=0}^{\infty}(\alpha)_{n}\frac{z^{n}}{n!}\right\rbrace   =\sum_{n=0}^{\infty}\frac{(\alpha)_{n}}{n!}D_{z}^{\delta,\mu;p;q;\lambda;m}\{z^{n}\} \\
  &=& \frac{z^{-\delta}}{\Gamma(-\delta)}\sum_{n=0}^{\infty}(\alpha)_{n}{B_{\mu}}(n+\frac{3}{2},-\delta +\frac{1}{2};p,q;\lambda;m)\frac{z^{n}}{n!}.
\end{eqnarray*}
Hence the result.
\end{proof}
Combining previous Lemmas, we obtain the generalized extended derivative of the product of analytic function with a power function.
\begin{theorem}\label{thm1}
Let $\Re(\delta)<0$. Suppose that a function $f(z)$ is analytic at the origin with its Maclaurin expansion given by $f(z)=\sum_{n=0}^{\infty}a_{n}z^{n}$, $(|z|<\rho)$ for some $\rho\in\mathbb{R}^{+}$. Then we have
\begin{equation}
    D_{z}^{\delta,\mu;p;q;\lambda;m}\{z^{\beta-1}f(z)\}=\sum_{n=0}^{\infty}a_{n}D_{z}^{\delta,\mu;p;q;\lambda;m}\{z^{\beta+n-1}\}=\frac{z^{\beta-\delta-1}}{\Gamma(-\delta)}\sum_{n=0}^{\infty}a_{n}B_{\mu}(\beta+n+\frac{1}{2},-\delta+\frac{1}{2};p;q;\lambda;m)z^{n}.
\end{equation}
\end{theorem}
A subsequent result can be given as follows
\begin{theorem}\label{thm2}
For $\Re(\delta)>\Re(\beta)>-\frac{1}{2}$, we have
\begin{equation}
 D_{z}^{\beta-\delta,\mu;p;q;\lambda;m}\{z^{\beta-1}(1-z)^{-\alpha}\}=\frac{z^{\delta-1}}{\Gamma(\delta-\beta)}B(\beta+\frac{1}{2},\delta-\beta+\frac{1}{2})F_{\mu}(\alpha,\beta+\frac{1}{2};\delta+1;z;q;\lambda;p;m)\;\;(|z|<1;\,\alpha\in\mathbb{C}).
\end{equation}
\end{theorem}
\begin{proof} The result is easily established by taking $f(z)=(1-z)^{-\alpha}$, so we have
\begin{eqnarray*}
  D_{z}^{\beta-\delta,\mu;p;q;\lambda;m}\{z^{\beta-1}(1-z)^{-\alpha}\} &=&D_{z}^{\beta-\delta,\mu;p;q;\lambda;m}\{z^{\beta-1}\sum_{k=0}^{\infty}(\alpha)_{k}\frac{z^{k}}{k!}\}  \\
  &=&\sum_{k=0}^{\infty}\frac{(\alpha)_{k}}{k!}D_{z}^{\beta-\delta,\mu;p;q;\lambda;m}\{z^{\beta+k-1}\} \\
  &=&\sum_{k=0}^{\infty}\frac{(\alpha)_{k}}{k!}\frac{B_{\mu}(\beta+k+\frac{1}{2},\delta-\beta+\frac{1}{2};p;q;\lambda;m)}{\Gamma(\delta-\beta)}z^{\delta+k-1}.
\end{eqnarray*}
By the expression \eqref{gbeta2}, we get
\begin{equation*}
  D_{z}^{\beta-\delta,\mu;p;q;\lambda;m}\{z^{\beta-1}(1-z)^{-\alpha}\}=\frac{z^{\delta-1}}{\Gamma(\delta-\beta)}B(\beta+\frac{1}{2},\delta-\beta+\frac{1}{2})F_{\mu}(\alpha,\beta+\frac{1}{2};\delta+1;z;q;\lambda;p;m).
\end{equation*}
\end{proof}
%The following makes straight connection between $F_{\mu}$ and $F_{1,\mu}$
\begin{theorem}\label{thm3} For $\Re(\delta)>\Re(\beta)>-\frac{1}{2}$, $\Re(\alpha)>0,$ $\Re(\gamma)>0$, $|az|<1$ and $|bz|<1$. Then, the following generating relation holds true
\begin{equation}
D_{z}^{\beta-\delta,\mu;p;q;\lambda;m}\{z^{\beta-1}(1-az)^{-\alpha}(1-bz)^{-\gamma}\}=\frac{z^{\delta-1}}{\Gamma(\delta-\beta)}B\left(\beta+\frac{1}{2},\delta-\beta+\frac{1}{2}\right)F_{1,\mu}(\beta+\frac{1}{2},\alpha,\gamma;\delta+1;az,bz;q;\lambda;p;m).
 \end{equation}
\end{theorem}
\begin{proof} By  applying the binomial Theorem to $(1-az)^{-\alpha}$ and $(1-bz)^{-\gamma}$ and making use of  Lemmas \ref{lem1} and \ref{lem2}, we obtain
\begin{eqnarray*}
  D_{z}^{\beta-\delta,\mu;p;q;\lambda;m}\{z^{\beta-1}(1-az)^{-\alpha}(1-bz)^{-\gamma}\} &=&D_{z}^{\beta-\delta,\mu;p;q;\lambda;m}\{z^{\beta-1}\sum_{k=0}^{\infty}\sum_{r=0}^{\infty}(\alpha)_{k}(\gamma)_{r}\frac{(az)^{k}}{k!}\frac{(bz)^{r}}{r!}\}  \\
  &=&\sum_{k,r=0}^{\infty}(\alpha)_{k}(\gamma)_{r}D_{z}^{\beta-\delta,\mu;p;q;\lambda;m}\{z^{\beta+k+r-1}\} \frac{a^{k}}{k!}\frac{b^{r}}{r!}\\
  &=&z^{\delta-1}\sum_{k,r=0}^{\infty}(\alpha)_{k}(\gamma)_{r}\frac{B_{\mu}(\beta+k+r+\frac{1}{2},\delta-\beta+\frac{1}{2};p;q;\lambda;m)}{\Gamma(\delta-\beta)}\frac{(az)^{k}}{k!}\frac{(bz)^{r}}{r!}.
\end{eqnarray*}
By using  \eqref{gbeta3}, we can get
\begin{equation*}
D_{z}^{\beta-\delta,\mu;p;q;\lambda;m}\{z^{\beta-1}(1-az)^{-\alpha}(1-bz)^{-\gamma}\}=\frac{z^{\delta-1}}{\Gamma(\delta-\beta)}B\left(\beta+\frac{1}{2},\delta-\beta+\frac{1}{2}\right)F_{1,\mu}(\beta+\frac{1}{2},\alpha,\gamma;\delta+1;az,bz;q;\lambda;p;m).
 \end{equation*}
\end{proof}
\begin{theorem}\label{thm4}For $\Re(\delta)>\Re(\beta)>-\frac{1}{2}$, $\Re(\alpha)>0$, $\Re(\gamma)>0$, $\Re(\tau)>0$, $|az|<1$, $|bz|<1$ and $|cz|<1$. \\
Then we have
\begin{equation}
D_{z}^{\beta-\delta,\mu;p;q;\lambda;m}\{z^{\beta-1}(1-az)^{-\alpha}(1-bz)^{-\gamma}(1-cz)^{-\tau}\}=\frac{z^{\delta-1}}{\Gamma(\delta-\beta)}B\left(\beta+\frac{1}{2},\delta-\beta+\frac{1}{2}\right)F_{D,\mu}^{3}(\beta+\frac{1}{2},\alpha,\gamma,\tau;\delta+1;az,bz;q;\lambda;p;m).
\end{equation}
\end{theorem}
\begin{proof}
 The proof is similar to that of  Theorem \ref{thm3}, it is sufficient to use the binomial Theorem for $(1-az)^{-\alpha}$, $(1-bz)^{-\gamma}$, $(1-cz)^{-\tau},$ then applying Lemmas \ref{lem1} and \ref{lem2}.
\end{proof}
\begin{theorem}\label{thm5} For $\Re(\delta)>\Re(\beta)>-\frac{1}{2}$, $\Re(\alpha)>0$, $\Re(\tau)>\Re(\gamma)>0$, $\left|\frac{x}{1-z}\right|<1$ and $|x|+|z|<1,$ we have
\begin{eqnarray}
&&D_{z}^{\beta-\delta,\mu;p;q;\lambda;m}\{z^{\beta-1}(1-z)^{-\alpha}F_{\mu}(\alpha,\gamma;\tau;\frac{x}{1-z};q;\lambda;p;m)\} \nonumber\\
&=&z^{\delta-1}\frac{B(\beta+\frac{1}{2},\delta-\beta+\frac{1}{2})}{\Gamma(\delta-\beta)}F_{2,\mu}(\alpha,\gamma,\beta+\frac{1}{2},\tau;\delta+1;x,z;q;\lambda;p;m).
  \end{eqnarray}
\end{theorem}
\begin{proof} By the binomial formula and according to Definition \ref{Defin}, we expand $z^{\beta-1}(1-z)^{-\alpha}F_{\mu}(\alpha,\gamma;\tau;\frac{x}{1-z};q;\lambda;p;m)$ to get
\begin{eqnarray*}
&&D_{z}^{\beta-\delta,\mu;p;q;\lambda;m}\left\{z^{\beta-1}(1-z)^{-\alpha}F_{\mu}(\alpha,\gamma;\tau;\frac{x}{1-z};q;\lambda;p;m)\right\} \\
  &=&D_{z}^{\beta-\delta,\mu;p;q;\lambda;m}\left\{z^{\beta-1}(1-z)^{-\alpha}\sum_{n=0}^{\infty}\frac{(\alpha)_{n}}{n!}\frac{B_{\mu}(\gamma+n,\tau-\gamma;q;\lambda;p;m)}{B(\gamma,\tau-\gamma)}\left(\frac{x}{1-z}\right)^{n} \right\} \\
&=&\sum_{n=0}^{\infty}(\alpha)_{n}\frac{B_{\mu}(\gamma+n,\tau-\gamma;q;\lambda;p;m)}{B(\gamma,\tau-\gamma)}D_{z}^{\beta-\delta,\mu;p;q;\lambda;m}\{z^{\beta-1}(1-z)^{-\alpha-n}\}\frac{x^{n}}{n!}.
\end{eqnarray*}
In order to exhibit $F_{2, \mu}$, we apply Theorem \ref{thm2} for $D_{z}^{\beta-\delta,\mu;p;q;\lambda;m}\{z^{\beta-1}(1-z)^{-\alpha-n}\}$ and substitute the extended hypergeometric function $F_{\mu}$ by its series representation, we obtain

\begin{eqnarray*}
&&  D_{z}^{\beta-\delta,\mu;p;q;\lambda;m}\{z^{\beta-1}(1-z)^{-\alpha}F_{\mu}(\alpha,\gamma;\tau;\frac{x}{1-z};q;\lambda;p;m)\} \\
  &&= \frac{z^{\delta-1}}{\Gamma(\delta-\beta)}{B(\beta+\frac{1}{2},\delta-\beta+\frac{1}{2})}\sum_{n,k=0}^{\infty}(\alpha)_{n+k}\frac{B_{\mu}(\gamma+n,\tau-\gamma;q;\lambda;p;m)}{B(\gamma,\tau-\gamma)} \times \frac{B_{\mu}(\beta+k+\frac{1}{2},\delta-\beta+\frac{1}{2};q;\lambda;p;m)}{B(\beta+\frac{1}{2},\delta-\beta+\frac{1}{2})}\frac{x^{n}z^{k}}{n!z!} \\
 &&= \frac{z^{\delta-1}}{\Gamma(\delta-\beta)}{B(\beta+\frac{1}{2},\delta-\beta+\frac{1}{2})}F_{2,\mu}(\alpha,\gamma,\beta+\frac{1}{2},\tau;\delta+1;x,z;q;\lambda;p;m).
\end{eqnarray*}
This completes the proof.
\end{proof}

\begin{proposition}[Mellin transform] The following expression holds true
\begin{eqnarray}
{\cal{M}}\{D_z^{\delta,\mu,p;q;\lambda;m}z^{\beta},p\to s\}&=&2^{s-1}z^{\beta-\delta}{\frac{1}{\sqrt{\pi}}}B\left(\beta+m(s+\frac{1}{2})+1,-\delta+m(s+\frac{1}{2})\right)\nonumber\\
&&\times\Gamma\left(\frac{s-\mu}{2}\right)\Gamma\left(\frac{s+\mu+1}{2}\right)
\Phi\left(\lambda,\frac{s+\mu+1}{2},q\right),
\end{eqnarray}
for $\Re(\mu)\geq 0$, $m>0$ and  $\Re(s)>\max\left\{\Re(\mu),-\frac{1}{2}-\frac{1}{m}-\frac{\Re(\beta)}{m},\frac{\Re(\delta)}{m}-\frac{1}{2}\right\rbrace $.
\end{proposition}
\begin{proof}
We can prove this result  by applying  Mellin transform and using Lemma \ref{lem1}.
\begin{eqnarray*}
{\cal{M}}\{D_z^{\delta,\mu,p;q;\lambda;m}z^{\beta},p\to s \} &=& \frac{1}{{\Gamma(-\delta)}}\int_{0}^{\infty}p^{s-1} z^{\beta-\delta}B_{\mu}(\beta+\frac{3}{2},-\delta+\frac{1}{2};p;q;\lambda;m)dp\\
&=&\frac{z^{\beta-\delta}}{{\Gamma(-\delta)}}\int_{0}^{\infty}p^{s-1} B_{\mu}(\beta+\frac{3}{2},-\delta+\frac{1}{2};p;q;\lambda;m)dp
\end{eqnarray*}
As the last integral is the Mellin transform of $B_{\mu}(\beta+\frac{3}{2},-\delta+\frac{1}{2};p;q;\lambda;m),$  the result immediately follows via Proposition \ref{proMelin}.
\end{proof}
\begin{proposition} The following expression holds true
    \begin{eqnarray}
{\cal{M}}\{D_z^{\delta,\mu,p;q;\lambda;m}(1-z)^{-\beta},p\to s\}&=&2^{s-1}z^{-\delta}\frac{1}{\sqrt{\pi}}B\left(m(s+\frac{1}{2})+1,-\delta +m(s+\frac{1}{2})\right)\Gamma\left(\frac{s-\mu}{2}\right)\Gamma\left(\frac{s+\mu+1}{2}\right)\nonumber\\
&&\times\Phi\left(\lambda,\frac{s+\mu+1}{2},q\right) {}_2F_1 (\beta, m(s+\frac{1}{2})+1;-\delta + m(2s+1)+1; z),
\end{eqnarray}
where $\Re(\mu)\geq 0$, $\Re(\delta)<0$, $m>0$, $|z|<1$, $\Re(s)>\max\left\{\Re(\mu),-\frac{1}{2}+\frac{1}{m},\frac{\delta}{m}-\frac{1}{2}\right\rbrace $ and ${}_2F_1 $ is  the well-known Gauss hypergeometric function.
\end{proposition}
\begin{proof}
    The result can be proved using the binomial Theorem for $(1-z)^{-\alpha}$ and the Mellin transform of the general term. Indeed,
\begin{eqnarray*}
 && {\cal{M}}\{D_{z}^{\delta,\mu;p;q;\lambda;m}\{(1-z)^{-\alpha}\},p\rightarrow s\}= {\cal{M}}\{D_{z}^{\delta,\mu;p;q;\lambda;m}\left\lbrace \sum_{n=0}^{\infty}(\alpha)_{n}\frac{z^{n}}{n!}\right\rbrace, p\rightarrow s\}   =\sum_{n=0}^{\infty}\frac{(\alpha)_{n}}{n!}{\cal{M}}\{D_{z}^{\delta,\mu;p;q;\lambda;m}{z^{n}},p\rightarrow s\} \}\\
  &&= \sum_{n=0}^{\infty}\frac{(\alpha)_{n}}{n!} 2^{s-1}z^{n-\delta}{\frac{1}{\sqrt{\pi}}}B\left(n +m(s+\frac{1}{2})+1,-\delta +m(s+\frac{1}{2})\right)\Gamma\left(\frac{s-\mu}{2}\right)\Gamma\left(\frac{s+\mu +1}{2}\right)\Phi\left(\lambda,\frac{s+\mu +1}{2},q\right).\\
  &&=2^{s-1}z^{-\delta}{\frac{1}{\sqrt{\pi}}}\Gamma\left(\frac{s-\mu}{2}\right)\Gamma\left(\frac{s+\mu +1}{2}\right)\Phi\left(\lambda,\frac{s+\mu +1}{2},q\right)\sum_{n=0}^{\infty}\frac{(\alpha)_{n}}{n!}B\left(n +m(s+\frac{1}{2})+1,-\delta +m(s+\frac{1}{2})\right)z^n\\
  &&=2^{s-1}z^{-\delta}\frac{1}{\sqrt{\pi}}B\left(m(s+\frac{1}{2})+1,-\delta +m(s+\frac{1}{2})\right)\Gamma\left(\frac{s-\mu}{2}\right)\Gamma\left(\frac{s+\mu+1}{2}\right)\Phi\left(\lambda,\frac{s+\mu+1}{2},q\right)\nonumber\\
&&\times {}_2F_1 (\beta, m(s+\frac{1}{2})+1;-\delta + m(2s+1)+1; z).
\end{eqnarray*}
\end{proof}

\section{Generating function involving the extended generalized Gauss hypergeometric function }
In this section, we establish  some generating functions for the generalized Gauss hypergeometric functions.

\begin{theorem}\label{thm6} Let $\Re(\beta)>0$ and $\Re(\gamma)>\Re(\alpha)>-\frac{1}{2}$. Then we have
\begin{equation}\label{aquea6}
  \sum_{n=0}^{\infty}\frac{(\beta)_{n}}{n!}F_{\mu}(\beta+n,\alpha+\frac{1}{2};\gamma+1;z;q;p;\lambda;m)t^{n}
  =(1-t)^{-\beta}F_{\mu}\left(\beta,\alpha+\frac{1}{2};\gamma+1;\frac{z}{1-t};q;p;\lambda;m\right),
  \end{equation}
where $|z|<\min\{1,|1-t|\}$.
\end{theorem}
\begin{proof} By considering the following elementary identity
\begin{equation*}
(1-z)^{-\beta}\left(1-\frac{t}{1-z}\right)^{-\beta}=(1-t)^{-\beta}\left(1-\frac{z}{1-t}\right)^{-\beta}
\end{equation*}
and expanding its left-hand side to give
\begin{equation}\label{series1}
    (1-z)^{-\beta}\sum_{n=0}^{\infty}\frac{(\beta)_{n}}{n!}\left(\frac{t}{1-z}\right)^n=(1-t)^{-\beta}\left(1-\frac{z}{1-t}\right)^{-\beta},\;\; \text{for }|t|<|1-z|.
\end{equation}
Multiplying both sides of \eqref{series1} by $z^{\alpha-1}$ and applying the extended Riemann-Liouville fractional derivative operator $D^{\alpha-\gamma;\mu;q;p;\lambda;m}$, we find
\begin{equation*}
D^{\alpha-\gamma;\mu;q;p;\lambda;m}\left\{\sum_{n=0}^{\infty}\frac{(\beta)_{n}t^{n}}{n!}
z^{\alpha-1}(1-z)^{-\beta-n}\right\}=D^{\alpha-\gamma;\mu;q;p;\lambda;m}\{(1-t)^{-\beta}z^{\alpha-1}(1-\frac{z}{1-t})^{-\beta}\}.
\end{equation*}
Uniform convergence of the involved series allows us to permute the summation and fractional derivative operator to get
\begin{equation}\label{series2}
\sum_{n=0}^{\infty}\frac{(\beta)_{n}}{n!}D^{\alpha-\gamma;\mu;q;p;\lambda;m}\{z^{\alpha-1}(1-z)^{-\beta-n}\}t^{n}=(1-t)^{-\beta}D^{\alpha-\gamma;\mu;q;p;\lambda;m}\{z^{\alpha-1}(1-\frac{z}{1-t})^{-\beta}\}.
\end{equation}
The result  easily follows using Theorem \ref{thm2}. %to both sides of \eqref{series2}.
\end{proof}
\begin{theorem}\label{thm7}Let $\Re(\beta)>0$, $\Re(\tau)>0$ and $\Re(\gamma)>\Re(\alpha)>-\frac{1}{2}$. Then we have
\begin{equation*}
    \sum_{n=0}^{\infty}\frac{(\beta)_{n}}{n!}F_{\mu}(\beta-n,\alpha+\frac{1}{2};\gamma+1;z;q;p;\lambda;m)t^{n}=(1-t)^{-\beta}
    F_{1,\mu}(\alpha+\frac{1}{2},\tau,\beta;\gamma+1;z;\frac{-zt}{1-t};q;p;\lambda;m),
\end{equation*}
where $|z|<1,$ $|t|<|1-z|$ and $|z||t|<|1-t|$.
\end{theorem}
\begin{proof} By considering the following identity
\begin{equation*}
    [1-(1-z)t]^{-\beta}=(1-t)^{-\beta}(1+\frac{zt}{1-t})^{-\beta},
\end{equation*}
and expanding its left-hand side as  power series, we get
\begin{equation*}
\sum_{n=0}^{\infty}\frac{(\beta)_{n}}{n!}(1-z)^{n}t^{n}=(1-t)^{-\beta}(1-\frac{-zt}{1-t})^{-\beta},\;\; \text{for }|t|<|1-z|.
\end{equation*}
Multiplying both sides by $z^{\alpha-1}(1-z)^{-\tau}$ and applying the definition of the extended Riemann-Liouville fractional derivative operator $D_{z}^{\alpha-\gamma;\mu;q;p;\lambda;m}$ on both sides, we find
\begin{eqnarray*}
   D_{z}^{\alpha-\gamma;\mu;q;p;\lambda;m}\left\{\sum_{n=0}^{\infty}\frac{(\beta)_{n}}{n!}z^{\alpha-1}(1-z)^{-\tau}(1-z)^{n}t^{n}\right\}
   =D_{z}^{\alpha-\gamma;\mu;q;p;\lambda;m}\left\{(1-t)^{-\beta}z^{\alpha-1}(1-z)^{-\tau}(1-\frac{-zt}{1-t})^{-\beta}\right\}.
\end{eqnarray*}
Interchanging the order of the summation and fractional derivative under the given conditions, we obtain
\begin{equation*}
\sum_{n=0}^{\infty}\frac{(\beta)_{n}}{n!}D^{\alpha-\gamma;\mu;q;p;\lambda;m}\{z^{\alpha-1}(1-z)^{-\tau+n}\}t^{n}=(1-t)^{-\beta}D^{\alpha-\gamma;\mu;q;p;\lambda;m}\left\{z^{\alpha-1}(1-z)^{-\tau}(1-\frac{z}{1-t})^{-\beta}\right\}.
 \end{equation*}
Finally, the desired result follows by Theorems \ref{thm2} and \ref{thm3}.
\end{proof}
\begin{theorem}\label{thm8}Let $\Re(\xi)>\Re(\upsilon)>-\frac{1}{2}$, $\Re(\gamma)>\Re(\alpha)>-\frac{1}{2}$ and $\Re(\beta)>0$. Then we have
\begin{eqnarray*}
  &&\sum_{n=0}^{\infty}\frac{(\beta)_{n}}{n!}F_{\mu}(\beta+n,\alpha+\frac{1}{2};\gamma+1;z;q;\lambda;p;m)F_{\mu}(-n,\upsilon+\frac{1}{2};\xi+1;u;q;\lambda;p;m)t^{n} \\
  &=&(1-t)^{-\beta}F_{2,\mu}\left(\beta,\alpha+\frac{1}{2},\upsilon+\frac{1}{2};\gamma+1,\xi+1;\frac{z}{1-t},\frac{-ut}{1-t};q;\lambda;p;m\right).
  \end{eqnarray*}
where $|z|<1$, $|\frac{1-u}{1-z}t|<1$ and $ |\frac{z}{1-t}|+|\frac{ut}{1-t}|<1.$
\end{theorem}
\begin{proof} By replacing $t$ by $(1-u)t$ in \eqref{aquea6}  and multiplying both sides of the resulting identity by $u^{\upsilon-1},$ we get
\begin{eqnarray}\label{series4}
 &&\sum_{n=0}^{\infty}\frac{(\beta)_{n}}{n!}F_{\mu}(\beta+n,\alpha+\frac{1}{2};\gamma+1;z;q;\lambda;p;m)u^{\upsilon-1}(1-u)^{n}t^{n} \nonumber\\
 &=&u^{\upsilon-1}[1-(1-u)t]^{-\beta}F_{\mu}\left(\beta,\alpha+\frac{1}{2};\gamma+1;\frac{z}{1-(1-u)t};q;\lambda;p;m\right),
\end{eqnarray}
where $\mathfrak{R}(\beta)>0$ and $\mathfrak{R}(\gamma)>\mathfrak{R}(\alpha)>-\frac{1}{2}$.

Next, applying the fractional derivative $D^{\upsilon-\xi,\mu;q;\lambda;p;m}$ to both sides of \eqref{series4} and changing the order of the summation and the fractional derivative under conditions $|z|<1$, $|\frac{1-u}{1-z}t|<1$ and $ |\frac{z}{1-t}|+|\frac{ut}{1-t}|<1$, yields
\begin{eqnarray*}
  &&\sum_{n=0}^{\infty}\frac{(\beta)_{n}}{n!}F_{\mu}(\beta+n,\alpha+\frac{1}{2};\gamma+1;z;q;\lambda;p;m)D^{\upsilon-\xi,\mu;q;\lambda;p;m}\{u^{\upsilon-1}(1-u)^{n}\}t^{n} \\
  &=&D^{\upsilon-\xi,\mu;q;\lambda;p;m}\left\{u^{\upsilon-1}[1-(1-u)t]^{-\beta}F_{\mu}\left(\beta,\alpha+\frac{1}{2};\gamma+1;\frac{z}{1-(1-u)t};q;\lambda;p;m\right)\right\},
\end{eqnarray*}

The last identity can be written as follows:
\begin{eqnarray*}
&&\sum_{n=0}^{\infty}\frac{(\beta)_{n}}{n!}F_{\mu}(\beta+n,\alpha+\frac{1}{2};\gamma+1;z;q;\lambda;p;m)D^{\upsilon-\xi,\mu;q;\lambda;p;m}\{u^{\upsilon-1}(1-u)^{n}\}t^{n}   \\
&=&(1-t)^{-\beta}D^{\upsilon-\xi,\mu;q;\lambda;p;m}\left\{u^{\upsilon-1}\left[1-\frac{-ut}{1-t}\right]^{-\beta}F_{\mu}\left(\beta+n,\alpha+\frac{1}{2};\gamma+1;\frac{\frac{z}{1-t}}{1-\frac{-ut}{1-t}};q;\lambda;p;m\right)\right\}.
\end{eqnarray*}
Thus, by using  Theorems \ref{thm2} and \ref{thm5} in the resulting identity, we obtain the desired result.
\end{proof}

%%%%%%%%%%%%%%%%%%%%%%%%%%%%%%%%%%%%%%%%%%%%%%%%%%%%%%%%%%%%%%%%%%%%%%%%%%%

\end{document}